\documentclass[reqno,11pt]{amsart}
\usepackage{amssymb,amsmath,enumerate,tikz,hyperref}
\usepackage{bbm,dsfont,soul}

\numberwithin{equation}{section}
\newtheorem{theorem}[equation]{Theorem}
\newtheorem{lemma}[equation]{Lemma}
\newtheorem{proposition}[equation]{Proposition}

\theoremstyle{definition}
\newtheorem{example}[equation]{Example}

\newcommand{\defeq}{\stackrel{\mathrm{def}}{=}}

\title[Enumeration of colored Dyck paths via partial Bell polynomials]{Enumeration of colored Dyck paths via \\ partial Bell polynomials}
\author{Daniel Birmajer}
\address{Department of Mathematics\\ Nazareth College\\ 4245 East Ave.\\ Rochester, NY 14618}
\author{Juan B. Gil}
\address{Penn State Altoona\\ 3000 Ivyside Park\\ Altoona, PA 16601}
\author{Peter R.W. McNamara}
\address{Department of Mathematics\\ Bucknell University\\ 1 Dent Drive\\ Lewisburg, PA 17837}
\author{Michael D. Weiner}
\address{Penn State Altoona\\ 3000 Ivyside Park\\ Altoona, PA 16601}

\begin{document}
\maketitle

\begin{abstract}
We consider a class of lattice paths with certain restrictions on their ascents and down steps and use them as building blocks to construct various families of Dyck paths. We let every building block $P_j$ take on $c_j$ colors and count all of the resulting colored Dyck paths of a given semilength. Our approach is to prove a recurrence relation of convolution type, which yields a representation in terms of partial Bell polynomials that simplifies the handling of different colorings. This allows us to recover multiple known formulas for Dyck paths and related lattice paths in an unified manner.
\end{abstract}

\section{Introduction}
\label{sec:intro}

A {\em Dyck path} of semilength $n$ is a lattice path in the first quadrant, which begins at the origin $(0,0)$, ends at $(2n,0)$, and consists of steps $(1,1)$ and $(1,-1)$. It is customary to encode an up-step $(1,1)$ with the letter $u$ and a down-step $(1,-1)$ with the letter $d$. Thus every Dyck path can be encoded by a corresponding {\em Dyck word} of $u$'s and $d$'s. We will freely pass from paths to words and vice versa. 

Much is known about Dyck paths and their connection to other combinatorial structures like rooted trees, noncrossing partitions, polygon dissections, Young tableaux, as well as other lattice paths. While there is a vast literature on the enumeration of Dyck paths and related combinatorial objects according to various statistics, for the scope of the present work, we only refer to the closely related papers \cite{AM08, Deutsch99, MS08}. For more information, the reader is referred to the general overview on lattice path enumeration written by C.~Krattenthaler in \cite[Chapter~10]{Handbook}.

\smallskip
For $a, b\in\mathbb{N}_0=\mathbb{N}\cup\{0\}$ with $a+b\not=0$ and $\mathbf{c}=(c_1,c_2,\dots)$ with $c_j\in\mathbb{N}_0$, we define
\begin{quote}
$\mathfrak{D}^{\mathbf{c}}_n(a,b)$ as the set of Dyck words of semilength $(a+b)n$ created from strings of the form $P_0=``d\,"$ and $P_j=``u^{(a+b)j}d^{b(j-1)+1}"$ for $j=1,\ldots,n$, such that each maximal $(a+b)j$-ascent substring $u^{(a+b)j}$ may be colored in $c_j$ different ways. We use $c_j=0$ if $(a+b)j$-ascents are to be avoided.  We will refer to the elements of $\mathfrak{D}^{\mathbf{c}}_n(a,b)$ as \emph{colored Dyck paths}.
\end{quote}
\smallskip

Note that if $a=1$, $b=0$, and $\mathbf{c}$ is the sequence of ones $\mathbf{c}=\mathbbm{1}=(1,1,\dots)$, then the building blocks take the form $P_0=``d\,"$, $P_j=``u^j d\,"$ for $j=1,\dots,n$, and $\mathfrak{D}^{\mathbbm{1}}_n(1,0)$ is just the set of regular Dyck words of semilength $n$.

In this paper, we are interested in counting the number of elements in $\mathfrak{D}^{\mathbf{c}}_n(a,b)$. For the sequence given by $y_n=\left|\mathfrak{D}^{\mathbf{c}}_n(a,b)\right|$, we prove a recurrence relation of convolution type (see Theorem~\ref{thm:yn_recurrence}) and give a representation of $y_n$ in terms of partial Bell polynomials in the elements of the sequence $\mathbf{c}=(c_1,c_2,\dots)$ (see Theorem~\ref{thm:Bell}). 

We conclude with several examples that illustrate the use of our formulas for various values of the parameters $a$ and $b$ as well as some interesting coloring choices.

\section{Enumeration of colored Dyck words}
\label{sec:enumeration}

Our technique for enumerating $\mathfrak{D}^{\mathbf{c}}_n(a,b)$ will be to show in Theorem~\ref{thm:yn_recurrence} and Proposition~\ref{thm:zn_recurrence} that the sequence $y_n=\left|\mathfrak{D}^{\mathbf{c}}_n(a,b)\right|$ satisfies the same initial condition and recurrence relation as a sequence $(z_n)$ involving Bell polynomials.  As a direct consequence, we get the promised enumeration of $\mathfrak{D}^{\mathbf{c}}_n(a,b)$ in terms of partial Bell polynomials (Theorem~\ref{thm:Bell}).

\begin{theorem} \label{thm:yn_recurrence}
For $a, b\in\mathbb{N}_0$ with $a+b\not=0$ and $\mathbf{c}=(c_1,c_2,\dots)$ with $c_j\in\mathbb{N}_0$, let $(y_n)$ be the sequence defined by $y_0=1$ and $y_n=\left|\mathfrak{D}^{\mathbf{c}}_n(a,b)\right|$ for $n\ge 1$. Then $y_n$ satisfies the recurrence
\begin{equation}\label{eq:recurrence}
 y_n = \sum_{\ell=1}^n c_{\ell} \sum_{i_1+\cdots + i_{a\ell+b} = n-\ell} y_{i_1} \cdots y_{i_{a\ell+b}},
\end{equation}
where each $i_j$ is a nonnegative integer.
\end{theorem}

\begin{proof}
We will prove \eqref{eq:recurrence} by showing that there is a bijection between the sets of objects counted by each side of the equation. The left-hand side counts colored Dyck words of semilength $(a+b)n$. The right-hand side counts tuples of the form
\[  (\ell, C; D_1, D_2, \ldots, D_{a\ell+b}),  \]
where 
\begin{itemize}
\item[$\circ$] $1\leq \ell \leq n$, 
\item[$\circ$] $C$ is a color from a choice of $c_{\ell}$ colors, 
\item[$\circ$] $D_j$ is a colored Dyck word of semilength $(a+b)i_j$, and 
\item[$\circ$] $i_1 + \cdots + i_{a\ell+b} = n-\ell$.  
\end{itemize}
From this tuple, we will construct a colored Dyck word $w$ of semilength $(a+b)n$ in the following fashion.

Due to the $\ell$ and $C$ appearing at the start of the tuple, we begin with $w = u^{(a+b)\ell}d^{b(\ell-1)+1}$ and color the substring $u^{(a+b)\ell}$ with the color $C$. We then append $D_1$, $D_2$, \ldots, $D_{a\ell+b}$ to $w$, separating each adjacent pair $(D_i, D_{i+1})$ by an additional copy of the letter $d$. We need to check that this map is well-defined, meaning that $w$ is a colored Dyck word of semilength $(a+b)n$.

Let us first check that $w$ contains equal numbers of the letters $u$ and $d$. Since the $D_i$ already satisfy this condition, we need 
\[ (a+b)\ell = (b(\ell-1)+1) + (a\ell+b-1), \]
which is true. Similar reasoning shows the ``Dyck'' property, i.e., that any prefix of $w$ has at least as many appearances of $u$ as of $d$. To determine the semilength of $w$, we count the number of appearances of $u$ as
\[ (a+b)\ell + (a+b)(n-\ell) = (a+b)n, \]
as desired. By construction, each maximal $(a+b)j$-ascent has an appropriate color and we conclude that $w$ is a colored Dyck word of semilength $(a+b)n$.  
 
To show that this map $f$ from the tuple to $w$ is a bijection, we argue that it has a well-defined inverse $g$.  Thus let $w$ be a colored Dyck path of semilength $(a+b)n$, and recall that $a$ and $b$ are fixed. The length $L$ and color of the ascent sequence at the beginning of $w$ determines $\ell=\frac{L}{a+b}$ and $C$ at the start of the tuple $g(w)$.  Let $w^1$ denote the word obtained from $w$ by removing this prefix $u^{(a+b)\ell}d^{b(\ell-1)+1}$ from $w$.  See Figure~\ref{fig:yn_recurrence_proof} for a schematic example.  We next wish to determine $D_1, \ldots, D_{a\ell+b}$ from $w^1$.  Let us say that $w^1$ has \emph{excess} $a\ell+b-1$, meaning that it has this many more copies of $d$ than of $u$.  Notice that this excess is nonnegative.

\begin{figure}[htbp]
\begin{center}
\begin{tikzpicture}[scale=0.45]
\draw [step=1,thin,gray!40] (0,0) grid (23,6);
\draw [very thick] (0,0) -- (5,5) -- (6,4);
\draw [very thick] (8,4) -- (9,3);
\draw [very thick] (13,3) -- (15,1);
\draw [very thick] (17,1) -- (18,0);
\draw [dashed] (6,4) -- (8,4);
\draw [dashed] (9,3) -- (13,3);
\draw [dashed] (15,1) -- (17,1);
\draw [very thick] (8,4) arc (0:180:1);
\draw [very thick] (13,3) arc (0:180:2);
\draw [very thick] (17,1) arc (0:180:1);
\draw [very  thick] (22,0) arc (0:180:2);
\draw [->] (0,0) -- (0,6);
\draw [->] (0,0) -- (23,0);
\draw (7,4.45) node {$D_1$};
\draw (11,4) node {$D_2$};
\draw (14.5,2.3) node {$D_3$};
\draw (16,1.45) node {$D_4$};
\draw (20,1) node {$D_5$};
\draw [<->, thick, blue] (6,2.5) -- node[above] {$r$} (9,2.5);
\draw [thick, dashed,red]  (6,-1.5) -- (6,6);
\draw [thick, dashed,red]  (9,-0.8) -- (9,6);
\draw [thick, red, ->] (6,-1.3) -- (7,-1.3);
\draw [thick, red, ->] (9,-0.6) -- (10,-0.6);
\draw [red] (6.6,-1.9) node {$w^1$};
\draw [red] (9.6,-1.2) node {$w^2$};
\end{tikzpicture}
\caption{A schematic example of determining $(\ell, C; D_1, D_2, \ldots, D_{a\ell+b})$ from $w$ as in the proof of Theorem~\ref{thm:yn_recurrence}, where the semicircles represent colored Dyck paths.  We have $L=5$, $a=5$, $b=0$, $\ell=1$, and $D_3$ is an empty word. For $i=1,2$, we see that $w^i$ is the portion of $w$ to the right of the corresponding dashed line.}
\label{fig:yn_recurrence_proof}
\end{center}
\end{figure}
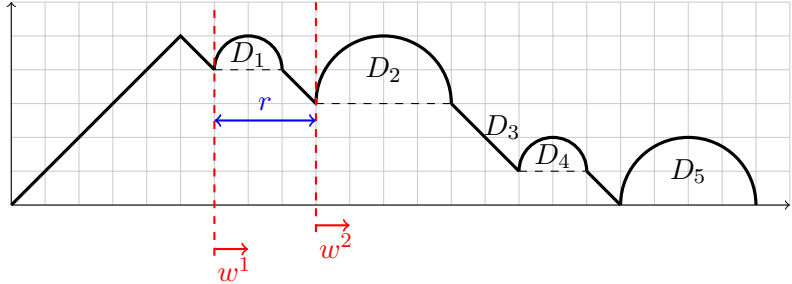

In short, we proceed by finding the smallest $r$ such that the suffix of $w^1$ corresponding to those letters strictly after position $r$ has excess one less than that of $w^1$.  Then we let $w^2$ be that suffix of $w$ and we let $D_1$ be the prefix of $w^1$ corresponding to those letters strictly before position $r$. In more detail, if the first letter of $w^1$ is $d$ then we get that $D_1$ is the empty word, and we let $w^2$ be the word obtained from $w^1$ by deleting this 1-letter prefix. Note that any extra copies of $d$, including the type just mentioned, correspond to the letters we used to separate adjacent pairs $(D_i, D_{i+1})$ in the forward map $f$. If the first letter of $w^1$ is $u$, then $D_1$ will be nonempty. We read off $D_1$ by following $w^1$ until we first reach a position $r$ where the number of appearances of $d$ in these first $r$ letters exceeds the number of appearances of $u$. If no such position $r$ exists, then it must be the case that $w^1$ has excess 0 and we let $D_1=w^1$.  Otherwise, we let $w^2$ be the word obtained from $w^1$ by deleting the first $r$ letters of $w^1$ and we let $D_1$ be the word (with colors) corresponding to the first $r-1$ letters of $w^1$. 

By the definition of $r$, $D_1$ has equal numbers of $u$'s and $d$'s, and it satisfies the Dyck property. Moreover, every maximal $(a+b)j$-ascent sequence is immediately followed by a $(b(j-1)+1)$-descent sequence because $w$ has this property and because these ascent lengths $(a+b)j$ are at least as large as their partnering descent lengths $(b(j-1)+1)$.  In other words, $D_1$ is a colored Dyck word. We continue in this exact manner to determine the full sequences $w^3, \ldots, w^{s}$ and $D_1, \ldots, D_s$ for some $s$.  Since each $w^i$ has excess one less than $w^{i-1}$, we deduce that $s-1 = a\ell+b-1$, and so $s=a\ell+b$, as desired.

Notice that the resulting tuple $g(w) = (\ell, C; D_1, D_2, \ldots, D_{a\ell+b})$ satisfies the properties in the four bullet points given at the beginning of this proof.  In particular, since $w$ has semilength $(a+b)n$, the total number of $u$'s in $D_1, D_2, \ldots, D_{a\ell+b}$ equals $(a+b)n - \ell(a+b)$, and so $i_1 + \cdots + i_{a\ell+b} = n-\ell$. We conclude that $g$ maps colored Dyck words of semilength $(a+b)n$ to tuples of the desired type.  Finally, one can readily observe that $g\circ f$ and $f \circ g$ both equal the identity map.
\end{proof}

\section{Representation in terms of partial Bell polynomials}
\label{sec:Bell}

Our goal for this section is to use the result of Theorem~\ref{thm:yn_recurrence} to give a formula for $y_n = \left|\mathfrak{D}^{\mathbf{c}}_n(a,b)\right|$ in terms of partial Bell polynomials.

For $a, b\in\mathbb{R}$ (not both $=0$) and $\mathbf{c}=(c_1,c_2,\dots)$, consider the sequence $(z_n)$ defined by
\begin{equation} \label{eq:abSequence}
 z_0=1, \quad  z_n = \sum_{k=1}^n \binom{an+bk}{k-1}\frac{(k-1)!}{n!} B_{n,k}(1!c_1, 2!c_2, \dots)
 \;\text{ for } n\ge 1,
\end{equation}
where $B_{n,k}$ denotes the $(n,k)$-th partial Bell polynomial defined as
\begin{equation*}
  B_{n,k}(x_1,\dots,x_{n-k+1})=\sum_{\alpha\in\pi(n,k)} \frac{n!}{\alpha_1! \cdots \alpha_{n-k+1}!}\left(\frac{x_1}{1!}\right)^{\alpha_1}\cdots \left(\frac{x_{n-k+1}}{(n-k+1)!}\right)^{\alpha_{n-k+1}}
\end{equation*}
with $\pi(n,k)$ denoting the set of multi-indices $\alpha\in{\mathbb N}_0^{n-k+1}$ such that $\alpha_1+\cdots+\alpha_{n-k+1}=k$ and $\alpha_1+2\alpha_2+\cdots+(n-k+1)\alpha_{n-k+1}=n$.  For more information on partial Bell polynomials, see \cite[Section~3.3]{Comtet}.

The sequence \eqref{eq:abSequence} satisfies the following convolution formula:

\begin{lemma}[{cf. \cite[Theorem~2.1]{BGW14a}}] \label{lem:r-convolution-general}
For $r, n \geq 1$, we have
\begin{equation*}
  z_n^{(r)} \defeq \sum_{m_1+\dots+m_r=n} \!\! z_{m_1}\cdots z_{m_r} 
 = r\sum_{k=1}^{n}\binom{a n+b k+r-1}{k-1} \frac{(k-1)!}{n!} B_{n,k}(1!c_1, 2!c_2, \dots).
\end{equation*}
\end{lemma}

\begin{proposition} \label{thm:zn_recurrence}
Suppose $a, b\in\mathbb{N}_0$. For  $n\ge 1$, the sequence $(z_n)$ defined by \eqref{eq:abSequence} satisfies the recurrence
\begin{equation} \label{eq:abRecurrence}
 z_n=\sum_{\ell=1}^n \; c_\ell \!\! \sum_{i_1+\dots+i_{a\ell+b}=n-\ell}z_{i_1}\dots z_{i_{a\ell+b}}
 = \sum_{\ell=1}^n c_\ell\, z_{n-\ell}^{(a\ell+b)},
\end{equation}
where each $i_j$ is a nonnegative integer and $z_0^{(an+b)}=1$.
\end{proposition}
\begin{proof}
By the previous lemma, omitting the argument of the Bell polynomials,
\begin{align*}
 \sum_{\ell=1}^{n-1} c_\ell\, z_{n-\ell}^{(a\ell+b)} 
 &= \sum_{\ell=1}^{n-1} c_\ell(a\ell+b) \sum_{k=1}^{n-\ell} \tbinom{an+b(k+1)-1}{k-1} \tfrac{(k-1)!}{(n-\ell)!} B_{n-\ell,k} \\
 &= \sum_{\ell=1}^{n-1} c_{n-\ell}\big(a(n-\ell)+b\big) 
 \sum_{k=1}^{\ell} \tbinom{an+b(k+1)-1}{k-1} \tfrac{(k-1)!}{\ell!} B_{\ell,k}\\
 &= \sum_{k=1}^{n-1} \tbinom{an+b(k+1)-1}{k-1} (k-1)! 
 \sum_{\ell=k}^{n-1} \tfrac{c_{n-\ell}(a(n-\ell)+b)}{\ell!} B_{\ell,k} 
\intertext{}
 &= \sum_{k=2}^{n} \tbinom{an+bk-1}{k-2} \tfrac{(k-2)!}{n!} 
 \sum_{\ell=k-1}^{n-1} \tfrac{n!}{\ell!}c_{n-\ell}(a(n-\ell)+b) B_{\ell,k-1} \\
 &= \sum_{k=2}^{n} \tbinom{an+bk-1}{k-2} \tfrac{(k-2)!}{n!} \sum_{\ell=k-1}^{n-1} 
 \Big(an\tbinom{n-1}{\ell} + b\tbinom{n}{\ell}\Big)(n-\ell)!c_{n-\ell} B_{\ell,k-1}.
\end{align*}

Now, using equations (11.11) and (11.12) in \cite[Theorem~11.12]{Charalambides}, one can easily verify the identities
\begin{align*}
 \sum_{\ell=k-1}^{n-1} an\tbinom{n-1}{\ell}(n-\ell)!c_{n-\ell} B_{\ell,k-1} &= anB_{n,k}, \\
 \sum_{\ell=k-1}^{n-1} b\tbinom{n}{\ell}(n-\ell)!c_{n-\ell} B_{\ell,k-1} &= bkB_{n,k},
\end{align*}
which imply
\begin{align*}
 \sum_{\ell=1}^{n-1} c_\ell\, z_{n-\ell}^{(a\ell+b)} 
 &= \sum_{k=2}^{n} \tbinom{an+bk-1}{k-2} \tfrac{(k-2)!}{n!} (an+bk)B_{n,k} 
   = \sum_{k=2}^{n} \tbinom{an+bk}{k-1} \tfrac{(k-1)!}{n!} B_{n,k}.
\end{align*}
Finally, by adding $c_n$ to each of these sums, we arrive at \eqref{eq:abRecurrence}.
\end{proof}

We now arrive at our main result.

\begin{theorem}\label{thm:Bell}
For $a, b\in\mathbb{N}_0$ with $a+b\not=0$ and $\mathbf{c}=(c_1,c_2,\dots)$, the sequence $y_n=\left|\mathfrak{D}^{\mathbf{c}}_n(a,b)\right|$ can be written as
\begin{equation}\label{eq:BellFormula}
 y_n = \sum_{k=1}^n \binom{an+bk}{k-1}\frac{(k-1)!}{n!} B_{n,k}(1!c_1, 2!c_2, \dots) \text{ for } n\ge 1.
\end{equation}
Moreover, the quantity $\binom{an+bk}{k-1}\frac{(k-1)!}{n!} B_{n,k}(1!c_1, 2!c_2, \dots)$ counts the number of Dyck paths in $\mathfrak{D}^{\mathbf{c}}_n(a,b)$ having exactly $k$ peaks.
\end{theorem}

\begin{proof}
Equation~\eqref{eq:BellFormula} is a direct consequence of Theorem~\ref{thm:yn_recurrence} and Proposition~\ref{thm:zn_recurrence}.  The second assertion follows by considering both sides of \eqref{eq:BellFormula} as polynomials in the $c_i$'s and by equating the terms of degree $k$.  
Indeed, note that $B_{n,k}(1!c_1, 2!c_2, \dots)$ contains as many monomials as there are partitions of $n$ into $k$ parts, and each such monomial has degree $k$ in the $c_i$'s.  On the other hand, each appearance of a $c_i$ in a monomial of $y_n$ corresponds to a coloring of a maximal ascent substring and therefore to a peak.   
\end{proof}

\section{Examples}
\label{sec:examples}

In this section we proceed to illustrate the use and versatility of the representation \eqref{eq:BellFormula}. The goal is to take advantage of the partial Bell polynomials to derive combinatorial formulas for the given enumerating sequence. 

First of all, as we mentioned in the introduction, $\mathfrak{D}^{\mathbbm{1}}_n(1,0)$ is nothing but the set of Dyck paths of semilength $n$. Recall that we are using the symbol $\mathbbm{1}$ to denote the sequence of ones $\mathbf{c}=(1,1,\dots)$.

\begin{example}(Narayana numbers) By \cite[Sec.~3.3, eqn.~(3h)]{Comtet} for example, 
\[ B_{n,k}(1!,2!,3!,\dots)=\frac{n!}{k!}\binom{n-1}{k-1} = \frac{(n-1)!}{(k-1)!}\binom{n}{k} \text{ for } n,k\ge 1, \] 
so Theorem~\ref{thm:Bell} gives the known fact that the number of Dyck paths of semilength $n$ with exactly $k$ peaks is given by
\begin{equation*}  
 \binom{n}{k-1}\frac{(k-1)!}{n!} B_{n,k}(1!, 2!, \dots) = \frac{1}{n} \binom{n}{k-1}\binom{n}{k},
\end{equation*}
the Narayana number $N(n,k)$.

In general, for any given parameters $a$ and $b$, and coloring sequence $\mathbf{c}$, the expressions 
\[  N^{\mathbf{c}}_{a,b}(n,k) = \binom{an+bk}{k-1}\frac{(k-1)!}{n!} B_{n,k}(1!c_1, 2!c_2, \dots) \] 
provide the appropriate analog of the Narayana numbers.
\end{example}

\begin{example}(Colored Motzkin paths)
It is known that the number of Motzkin paths of length $n$ is the same as the number of Dyck words of semilength $n$ that avoid $uuu$ (via the bijection $u^2d\to u$,  $d\to d$, and $ud\to h$, where $h$ denotes a horizontal step (1,0)). Thus, for $n\ge 1$, the number of Motzkin $n$-paths whose horizontal steps admit $c_1$ colors and whose up steps admit $c_2$ colors is given by 
\begin{align*}
 y_n &= \sum_{k=1}^n \binom{n}{k-1}\frac{(k-1)!}{n!} B_{n,k}(1!c_1, 2!c_2,0,\dots) \\
 &= \sum_{k=\lceil \frac{n}{2}\rceil}^n \binom{n}{k-1}\frac{(k-1)!}{n!} \frac{n!}{k!} \binom{k}{n-k} c_1^{2k-n}c_2^{n-k} \\
 &= \sum_{k=\lceil \frac{n}{2}\rceil}^n \frac{1}{n+1} \binom{n+1}{k} \binom{k}{n-k} c_1^{2k-n}c_2^{n-k} \\
 &= \sum_{k=0}^{\lfloor \frac{n}{2} \rfloor} \frac{1}{n+1} \binom{n+1}{n-k} \binom{n-k}{k} c_1^{n-2k}c_2^k \\
 &= \sum_{k=0}^{\lfloor \frac{n}{2} \rfloor} \binom{n}{2k} C_k\,  c_1^{n-2k}c_2^k,
\end{align*}
where $C_k$ denotes the Catalan number $\frac{1}{k+1}\binom{2k}{k}$.  Letting $c_1=c_2=1$ gives one of the better-known expressions for the Motzkin numbers.
\end{example}

\begin{example}(Schr\"oder numbers)
The numbers in the sequence \cite[A001003]{Sloane} are called little Schr\"oder numbers and are known to count (among other things) Dyck paths in which the interior vertices of the ascents admit two colors, that is, Dyck paths in which a maximal $j$-ascent may be colored in $2^{j-1}$ different ways. The number $y_n$ of such colored paths of semilength $n$ can be obtained from \eqref{eq:BellFormula} with $a=1$, $b=0$, and $\mathbf{c} = (1,2,2^2,\dots)$. Thus
\begin{align*}
 y_n &= \sum_{k=1}^n \binom{n}{k-1}\frac{(k-1)!}{n!} B_{n,k}(1!\cdot 1,\ 2!\cdot 2,\ 3!\cdot 2^2, \dots) \\
 &= \sum_{k=1}^n \binom{n}{k-1}\frac{(k-1)!}{n!} 2^{n-k} B_{n,k}(1!, 2!, \dots)\\
& = \sum_{k=1}^n \frac{1}{n} \binom{n}{k-1} \binom{n}{k} 2^{n-k} = \sum_{k=1}^n N(n,k)\, 2^{n-k}.
\end{align*}
\end{example}

\begin{example}($m$-ary paths)
For $m\in\mathbb{N}$ we consider the set $\mathfrak{D}^{\mathbbm{1}}_n(m,0)$ of Dyck words of semilength $mn$ created from strings of the form $P_0=d$ and $P_j=u^{mj}d$ for $j=1,\ldots,n$. 

The elements of $\mathfrak{D}^{\mathbbm{1}}_n(m,0)$ are in one-to-one correspondence with the elements of the set $\mathfrak{L}_n(m)$ of $m$-ary paths of length $(m+1)n$, i.e., lattice paths in the first quadrant from $(0,0)$ to $((m+1)n,0)$ with steps $(1,m)$ or $(1,-1)$. Here is an example for $m=2$:

\medskip
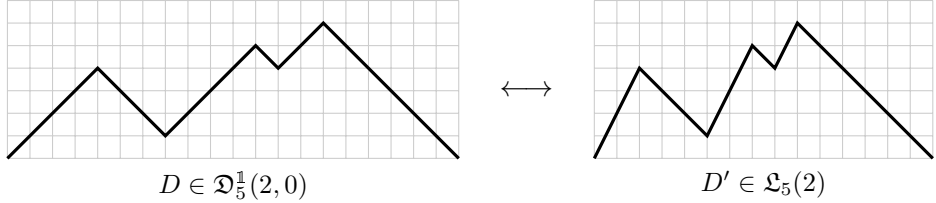
\begin{center}
\begin{tikzpicture}[scale=0.3]
\begin{scope}
\draw [step=1,thin,gray!40] (0,0) grid (20,7);
\draw [very thick] (0,0) -- (4,4) -- (7,1) -- (11,5) -- (12,4) -- (14,6) -- (20,0);
\node[below=1pt] at (10,0) {\small $D \in \mathfrak{D}^{\mathbbm{1}}_5(2,0)$};
\end{scope}

\node at (23,3) {$\longleftrightarrow$};

\begin{scope}[xshift=740]
\draw [step=1,thin,gray!40] (0,0) grid (15,7);
\draw [very thick] (0,0) -- (2,4) -- (5,1) -- (7,5) -- (8,4) -- (9,6) -- (15,0);
\node[below=1pt] at (7.5,0) {\small $D' \in \mathfrak{L}_5(2)$};
\end{scope}
\end{tikzpicture}
\end{center}

By equation \eqref{eq:BellFormula}, the  sequence $y_n = \left|\mathfrak{L}_n(m)\right| = \left|\mathfrak{D}^{\mathbbm{1}}_n(m,0)\right|$ is given by 
\begin{align*}
 y_n &= \sum_{k=1}^n \binom{mn}{k-1}\frac{(k-1)!}{n!} B_{n,k}(1!, 2!, \dots) \\
 &= \sum_{k=1}^n \frac{1}{k} \binom{mn}{k-1} \binom{n-1}{k-1} = \sum_{k=1}^n \frac{1}{mn+1} \binom{mn+1}{k}\binom{n-1}{n-k}, 
\end{align*}
which by Vandermonde's identity becomes
\begin{equation*}
 y_n = \frac{1}{mn+1} \binom{(m+1)n}{n}.
\end{equation*}
Moreover, the number of such paths with exactly $k$ peaks is given by the expression
\[  N^{\mathbbm{1}}_{m,0}(n,k) 
   = \frac{1}{k} \binom{mn}{k-1} \binom{n-1}{k-1} = \frac{1}{n} \binom{mn}{k-1} \binom{n}{k}. \]
These formulas are consistent with \cite[Corollary~4.12]{DNS15}. Clearly, Theorem~\ref{thm:Bell} also provides formulas for other choices of the coloring sequence $\mathbf{c}$. 
\end{example}

The next three examples illustrate simple connections with other types of lattice paths.
\begin{example}(\cite[A052709]{Sloane})
If $a=0$, $b=2$, and $\mathbf{c}=(1,1,0,0,\dots)$, the set $\mathfrak{D}^{\mathbf{c}}_n(0,2)$ consists of Dyck words of semilength $2n$ created from strings of the form $P_0=d$, $P_1=u^2d$, and $P_2=u^4d^3$. With the simple map $d\to (1,-1)$, $u^2d\to (1,1)$, and $u^4d^3\to (3,1)$, we get a one-to-one correspondence between $\mathfrak{D}^{\mathbf{c}}_n(0,2)$ and the set $\mathfrak{L}_n(0,2)$ of lattice paths in the first quadrant from $(0,0)$ to $(2n,0)$ with steps $(1,1)$, $(1,-1)$, or $(3,1)$. 

\medskip
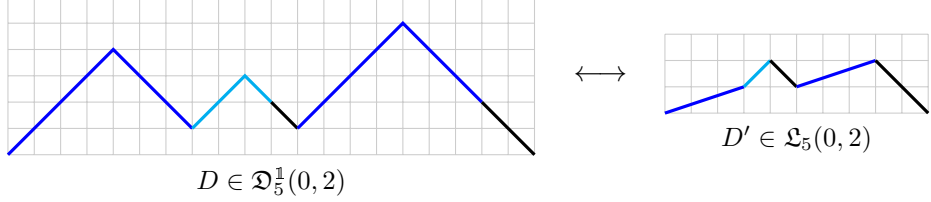
\begin{center}
\begin{tikzpicture}[scale=0.35]
\begin{scope}
\draw [step=1,thin,gray!40] (0,0) grid (20,6);
\draw[blue] [very thick] (0,0) -- (4,4) -- (7,1);
\draw[cyan] [very thick] (7,1) -- (9,3) -- (10,2);
\draw[black] [very thick] (10,2) -- (11,1);
\draw[blue] [very thick] (11,1) -- (15,5) -- (18,2);
\draw[black] [very thick] (18,2) -- (20,0);
\node[below=1pt] at (10,0) {\small $D \in \mathfrak{D}^{\mathbbm{1}}_5(0,2)$};
\end{scope}

\node at (22.5,3) {$\longleftrightarrow$};

\begin{scope}[xshift=710,yshift=45]
\draw [step=1,thin,gray!40] (0,0) grid (10,3);
\draw[blue] [very thick] (0,0) -- (3,1);
\draw[cyan] [very thick] (3,1) -- (4,2);
\draw[black] [very thick] (4,2) -- (5,1);
\draw[blue] [very thick] (5,1) -- (8,2);
\draw[black] [very thick] (8,2) -- (10,0);
\node[below=1pt] at (5,0) {\small $D' \in \mathfrak{L}_5(0,2)$};
\end{scope}
\end{tikzpicture}
\end{center}

By means of \eqref{eq:BellFormula}, we then get that $y_n = |\mathfrak{L}_n(0,2)| = \left|\mathfrak{D}^{\mathbf{c}}_n(0,2)\right|$ satisfies
\begin{equation*}
 y_n = \sum_{k=1}^n \binom{2k}{k-1}\frac{(k-1)!}{n!} B_{n,k}(1!, 2!,0,\dots) 
 = \sum_{k=\lceil \frac{n}{2}\rceil}^n \frac{1}{k} \binom{2k}{k-1} \binom{k}{n-k}.
\end{equation*}
\end{example}

\begin{example}(\cite[A186997]{Sloane})
If $a=1$, $b=2$, and $\mathbf{c}=(1,1,0,0,\dots)$, the set $\mathfrak{D}^{\mathbf{c}}_n(1,2)$ consists of Dyck words of semilength $3n$ created from strings of the form $P_0=d$, $P_1=u^3d$, and $P_2=u^6d^3$. With the simple map $d\to (1,-1)$, $u^3d\to (1,2)$, and $u^6d^3\to (3,3)$, we get a one-to-one correspondence between $\mathfrak{D}^{\mathbf{c}}_n(1,2)$ and the set $\mathfrak{L}_n(1,2)$ of lattice paths in the first quadrant from $(0,0)$ to $(3n,0)$ with steps $(1,2)$, $(1,-1)$, or $(3,3)$.

\medskip
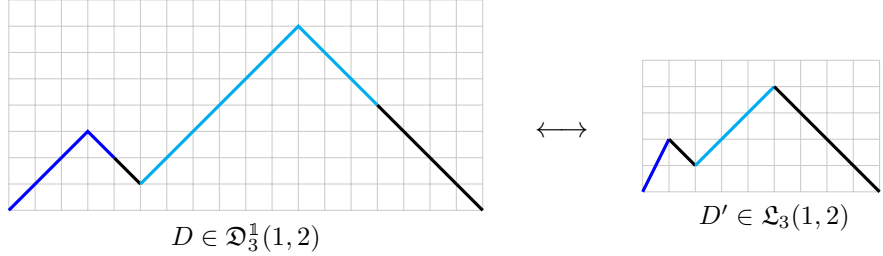
\begin{center}
\begin{tikzpicture}[scale=0.35]
\begin{scope}
\draw [step=1,thin,gray!40] (0,0) grid (18,8);
\draw[blue] [very thick] (0,0) -- (3,3) -- (4,2);
\draw[black] [very thick] (4,2) -- (5,1);
\draw[cyan] [very thick] (5,1) -- (11,7) -- (14,4);
\draw[black] [very thick] (14,4) -- (18,0);
\node[below=1pt] at (9,0) {\small $D \in \mathfrak{D}^{\mathbbm{1}}_3(1,2)$};
\end{scope}

\node at (21,3) {$\longleftrightarrow$};

\begin{scope}[xshift=685,yshift=20]
\draw [step=1,thin,gray!40] (0,0) grid (9,5);
\draw[blue] [very thick] (0,0) -- (1,2);
\draw[black] [very thick] (1,2) -- (2,1);
\draw[cyan] [very thick] (2,1) -- (5,4);
\draw[black] [very thick] (5,4) -- (9,0);
\node[below=1pt] at (5,0) {\small $D' \in \mathfrak{L}_3(1,2)$};
\end{scope}
\end{tikzpicture}
\end{center}

Again, by means of \eqref{eq:BellFormula}, we get that $y_n = |\mathfrak{L}_n(1,2)| = \left|\mathfrak{D}^{\mathbf{c}}_n(1,2)\right|$ satisfies
\begin{equation*}
 y_n = \sum_{k=1}^n \binom{n+2k}{k-1}\frac{(k-1)!}{n!} B_{n,k}(1!, 2!,0,\dots) 
 = \sum_{k=\lceil \frac{n}{2}\rceil}^n \frac{1}{k} \binom{n+2k}{k-1} \binom{k}{n-k}.
\end{equation*}
\end{example}

\begin{example}($\frac32$-Dyck paths)
In the context of generalized Dyck languages with only two letters, Duchon \cite{Duchon} studied rational Dyck paths and suggests the need for colored Dyck words. In particular, he considered the set of Dyck words with slope $\frac32$ and length $5n$, which can be visualized as generalized Dyck paths starting at $(0,0)$ and ending at $(2n,3n)$, without crossing the line $y=\frac32 x$. For example, for $n=2$,

\medskip
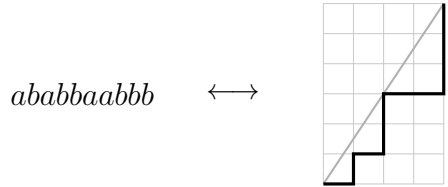
\begin{center}
\begin{tikzpicture}[scale=0.4]
\node at (-8,3) {\large $ababbaabbb$};
\node at (-3,3) {$\longleftrightarrow$};
\draw [step=1,thin,gray!40] (0,0) grid (4,6);
\draw [gray!60, thick] (0,0) -- (4,6);
\draw [very thick] (0,0) -- (1,0) -- (1,1) -- (2,1) -- (2,3) -- (4,3) -- (4,6);
\end{tikzpicture}
\end{center}

We denote this set by $\mathcal{D}_{3/2}(5n)$. In op.\ cit.\ Duchon proved that the number of factor-free elements of $\mathcal{D}_{3/2}(5n)$ is given by $C_{n-1} + C_n$, where $C_n$ is the $n$-th Catalan number.\footnote{A word in a language $L$ is said to be factor-free  if it has no proper factor in $L$.} Moreover, for $d_n=\left|\mathcal{D}_{3/2}(5n)\right|$, he gives the formula
\begin{equation*}
 d_n = \sum_{j=0}^n \frac{1}{5n+j+1} \binom{5n+1}{n-j} \binom{5n+2j}{j}.
\end{equation*}
This is sequence A060941 in \cite{Sloane}.

It turns out that these numbers may also be generated by counting the elements of $\mathfrak{D}^{\mathbf{c}}_n(5,0)$ with coloring sequence $\mathbf{c}=(C_{j-1}+C_j)_{j\ge 1}$. In other words, there is a bijection between $\mathcal{D}_{3/2}(5n)$ and the set of Dyck words of semilength $5n$ created from strings of the form $P_0=``d\,"$ and $P_j=``u^{5j}d"$ for $j=1,\ldots,n$, such that each maximal ascent $u^{5j}$ is colored by a factor-free Dyck word with slope $\frac32$ and length $5j$. 

Consequently, since $d_n = y_n = \left|\mathfrak{D}^{\mathbf{c}}_n(5,0)\right|$, Theorem~\ref{thm:Bell} gives the alternative formula
\begin{equation*}
  d_n  = \sum_{k=1}^n \binom{5n}{k-1}\frac{(k-1)!}{n!} B_{n,k}(1!(C_0+C_1), 2!(C_1+C_2),\dots).
\end{equation*}
Finally, since $j!(C_{j-1}+C_j) = (2j-2)_{j-1} + (2j)_{j-1}$, we can use the second identity in \cite[Example~3.2]{WW09} with $a=2$, $b=-1$, and $c=2$ to obtain
\begin{align*}
  d_n  &= \sum_{k=1}^n \binom{5n}{k-1} 
  \sum_{j=0}^k \frac{(-1)^{k-j}}{k} \binom{k}{j} (2j-k) \frac{(2j-k+2n-1)_{n-1}}{n!} \\
  &= \sum_{k=1}^n \binom{5n}{k-1} \sum_{j=0}^k \frac{(-1)^{k-j}}{nk} \binom{k}{j} (2j-k)\binom{2j-k+2n-1}{n-1} \\
  &= \sum_{k=1}^n \binom{5n}{k-1} \sum_{j=0}^k \frac{(-1)^{j}}{n} \left[\binom{k-1}{j} - \binom{k-1}{j-1}\right]
  \binom{2n+k-2j-1}{n-1}.
\end{align*}
\end{example}


\end{document}